\documentclass[reqno,oneside]{amsart}
\usepackage{hyperref}
\usepackage{geometry}
\usepackage[ansinew]{inputenc}
\usepackage{graphicx}
\usepackage{amsmath}
\usepackage{amsthm}
\usepackage{amssymb, color}%
\usepackage[numbers, square]{natbib}
\usepackage{mathrsfs}
\usepackage{bbm}
\usepackage{tikz}
\usepackage{mathptmx}
\usepackage{hyperref}

\usepackage{blindtext}
\usepackage{enumitem}

\makeatletter
\def\@tocline#1#2#3#4#5#6#7{\relax
  \ifnum #1>\c@tocdepth 
  \else
    \par \addpenalty\@secpenalty\addvspace{#2}%
    \begingroup \hyphenpenalty\@M
    \@ifempty{#4}{%
      \@tempdima\csname r@tocindent\number#1\endcsname\relax
    }{%
      \@tempdima#4\relax
    }%
    \parindent\z@ \leftskip#3\relax \advance\leftskip\@tempdima\relax
    \rightskip\@pnumwidth plus4em \parfillskip-\@pnumwidth
    #5\leavevmode\hskip-\@tempdima
      \ifcase #1
       \or\or \hskip 1em \or \hskip 2em \else \hskip 3em \fi%
      #6\nobreak\relax
    \dotfill\hbox to\@pnumwidth{\@tocpagenum{#7}}\par
    \nobreak
    \endgroup
  \fi}
\makeatother

\usepackage{fancybox}

\newlength{\querylen}
\newcommand{\mmp}{\mathbb{P}}

\newcommand{\me}{\mathbb{E}}
\newcommand{\E}{\mathbb{E}}

\newcommand{\mr}{\mathbb{R}}
\newcommand{\mn}{\mathbb{N}}

\DeclareMathOperator{\1}{\mathbbm{1}}

\newtheorem{thm}{Theorem}[section]
\newtheorem{lemma}[thm]{Lemma}

\newtheorem{cor}[thm]{Corollary}

\newtheorem{assertion}[thm]{Proposition}
\theoremstyle{definition}

\theoremstyle{remark}
\newtheorem{rem}[thm]{Remark}

\begin{document}
\title[Renewal theory for iterated perturbed random walks]{Renewal theory for iterated perturbed random walks on a general branching process tree: intermediate generations}

\author{Vladyslav Bohun}
\address{Faculty of Computer Science and Cybernetics, Taras Shevchenko National University of Kyiv, Kyiv, Ukraine}
\email{vladyslavbogun@gmail.com}

\author{Alexander Iksanov}
\email{iksan@univ.kiev.ua}

\author{Alexander Marynych}
\email{marynych@unicyb.kiev.ua}

\author{Bohdan Rashytov}
\email{mr.rashytov@gmail.com}

\begin{abstract}
Let $(\xi_k,\eta_k)_{k\in\mn}$ be independent identically distributed random vectors with arbitrarily dependent positive components. We call a (globally) perturbed random walk a random sequence $(T_k)_{k\in\mn}$ defined by $T_k:=\xi_1+\cdots+\xi_{k-1}+\eta_k$ for $k\in\mn$. Further, by an iterated perturbed random walk is meant the sequence of point processes defining the birth times of individuals in subsequent generations of a general branching process provided that the birth times of the first generation individuals are given by a perturbed random walk. For $j\in\mn$ and $t\geq 0$, denote by $N_j(t)$ the number of the $j$th generation individuals with birth times $\leq t$. In this article we prove counterparts of the classical renewal-theoretic results (the elementary renewal theorem, Blackwell's theorem and the key renewal theorem) for $N_j(t)$ under the assumption that $j=j(t)\to\infty$ and $j(t)=o(t^{2/3})$ as $t\to\infty$. According to our terminology, such generations form a subset of the set of intermediate generations.
\end{abstract}

\keywords{Convolution, general branching process, key renewal theorem, perturbed random walk, renewal theory} 

\subjclass[2010]{Primary: 60K05, 60J80; secondary: 60G05}

\maketitle

\tableofcontents

\section{Introduction}\label{Sect1}

The classical renewal theory is an area of applied probability dealing with nondecreasing standard random walks and various derived processes like {\it renewal process, first passage time, overshoot, undershoot} etc. A good overview of the renewal theory can be found in \cite{Asmussen:2003}, \cite{Resnick:2002} and more recent accounts \cite{Iksanov:2016} and \cite{Mitov+Omey:2014}.

Assume that a general branching process (a.k.a. Crump--Mode--Jagers branching process) is generated by a standard random walk $S^{(1)}$ with nonnegative steps. Clearly, the random sequence $S^{(j)}$ defined by the birth times in the $j$th generation of the process ($j\geq 2$) is much more complicated than the standard random walk $S^{(1)}$ defining the birth times in the $1$st generation. It is natural to call $(S^{(j)})_{j\geq 2}$ {\it iterated standard random walk on a general branching process tree}. This should not be confused with iterated renewal processes treated in \cite{Sen:1981}. In this paper we initiate a systematic study of $S^{(j)}$ for $j\geq 2$ and its derived processes, our primary purpose being obtaining counterparts of the classical renewal-theoretic results. Actually, our setting will be a bit more general than that outlined above. We shall develop elements of renewal theory for iterated {\it perturbed} random walks rather than standard random walks, thereby making our results more general.

Now it is time to set the scene. Let $(\xi_i, \eta_i)_{i\in\mn}$ be independent copies of a $\mr^2$-valued random vector $(\xi, \eta)$ with arbitrarily dependent components. Denote by $(S_i)_{i\geq 0}$ the zero-delayed standard random walk with increments $\xi_i$ for $i\in\mn$, that is, $S_0:=0$ and $S_i:=\xi_1+\cdots+\xi_i$ for $i\in\mn$. Define
\begin{equation*}
T_i:=S_{i-1}+\eta_i,\quad i\in \mn.
\end{equation*}
The sequence $T:=(T_i)_{i\in\mn}$ is called {\it perturbed random walk} (PRW, in short). A survey of various results for the so defined PRWs can be found in the book \cite{Iksanov:2016}. An incomplete list of more recent papers addressing various aspects of the PRWs includes \cite{Alsmeyer+Iksanov+Marynych:2017, Duchamps+Pitman+Tang:2019, Iksanov+Pilipenko+Samoilenko:2017, Pitman+Tang:2019, Pitman+Yakubovich:2019, Rashytov:2018}.

In what follows we assume that $\xi$ and $\eta$ are almost surely (a.s.) positive. Put $N(t):=\sum_{i\geq 1}\1_{\{T_i\leq t\}}$ and $V(t):=\me N(t)$ for $t\geq 0$.  It is clear that
\begin{equation}\label{equ}
V(t)=\me U((t-\eta)_{+})=(U\ast G)(t)=\int_{[0,\,t]}U(t-y){\rm d}G(y),\quad t\geq 0,
\end{equation}
where, for $t\geq 0$, $U(t):=\sum_{i\geq 0}\mmp\{S_i\leq t\}$ is the renewal function and $G(t)=\mmp\{\eta\leq t\}$. As usual, $x_{+}:=\max(x,0)$. Here and in what follows we denote by $u \ast v$ the Lebesgue--Stieltjes convolution of two functions $u,v$ of locally bounded variation. We also use the notation $u^{\ast(j)}$, $j\in\mn$, for the $j$th convolution power of $u$.

Now we provide more details about the construction of a general branching process (already mentioned at the beginning of the section) in the special case it is generated by $T$. At time $0$ there is one individual, the ancestor. The ancestor produces offspring (the first generation) with birth times given by the points of $T$. The first generation
produces the second generation. The shifts of birth times of the second generation
individuals with respect to their mothers' birth times are distributed according to copies of $T$, and for different mothers these copies are independent. The second generation produces the
third one, and so on. All individuals act independently of each other.

For $t\geq 0$ and $j\in\mn$, denote by $N_j(t)$ the number of the $j$th generation individuals with birth times $\leq t$ and put $V_j(t):=\me N_j(t)$, and $V(t):=0$ for $t<0$. Then $N_1(t)=N(t)$, $V_1(t)=V(t)$ and
$$
V_j(t)=(V_{j-1}\ast V)(t)=\int_{[0,\,t]} V_{j-1}(t-y){\rm d}V(y),\quad j\geq 2,\quad t\geq 0.
$$
The basic decomposition that sheds light on the properties of $N_j:=(N_j(t))_{t\geq 0}$ 
and also demonstrates its recursive structure is
\begin{equation}\label{basic1232}
N_j(t)=\sum_{r\geq 1}N^{(r)}_{j-1}(t-T_r)\1_{\{T_r\leq t\}}=\sum_{k\geq 1}N^{(k)}_1(t-T^{(j-1)}_k)\1_{\{T^{(j-1)}_k\leq t\}} ,\quad j\geq 2,\quad t\geq 0,
\end{equation}
where $N_{j-1}^{(r)}(t)$ is the number of successors in the $j$th generation with birth times within $[T_r,t+T_r]$ of the first generation individual with birth time $T_r$; $T^{(j-1)}:=(T^{(j-1)}_k)_{k\geq 1}$ is some enumeration of the birth times in the $(j-1)$th generation; $N_1^{(k)}(t)$ is the number of children in the $j$th generation with birth times within $[T^{(j-1)}_k,t+T^{(j-1)}_k]$ of the $(j-1)$th generation individual with birth time $T^{(j-1)}_k$. By the branching property, $(N_{j-1}^{(1)}(t))_{t\geq 0}$, $(N_{j-1}^{(2)}(t))_{t\geq 0},\ldots$ are independent copies of $N_{j-1}$ which are also independent of $T$, and $(N_1^{(1)}(t))_{t\geq 0}$, $(N_1^{(2)}(t))_{t\geq 0},\ldots$ are independent copies of $(N(t))_{t\geq 0}$ which are also independent of $T^{(j-1)}$.  Note that, for $j\geq 2$, $N_j$ is a particular instance of a random process with immigration at random times (the term was introduced in \cite{Dong+Iksanov:2020}, see also \cite{Iksanov+Rashytov:2020}).

Our motivation behind introducing the iterated perturbed random walks is at least three-fold.

\begin{enumerate}[leftmargin=*]
\item [1)] For each integer $j\geq 2$, the sequence $T^{(j)}$ and the process $N_j$ are a natural generalization of the perturbed random walk $T$ and the counting process $(N(t))_{t\geq 0}$. It is interesting to investigate to which extent the renewal-theoretic properties of $T$ and $(N(t))$ are inherited by $T^{(j)}$ and $N_j$.
Thus, the activity undertaken in the present article can be thought of as the development of renewal theory for the iterated perturbed random walks.

\item[2)] The sequence $(T^{(j)})_{j\in\mn}$ is a particular instance of a branching random walk in which the first generation point process is $(N(t))_{t\geq 0}$, the counting process of a perturbed random walk. Alternatively, and this is our preferable viewpoint, for $j\in\mn$, $T^{(j)}$ can be interpreted as the sequence of birth times in the $j$th generation of a general branching process. 
Therefore, the results of the present article contribute towards better understanding of how the births occur within a particular generation.  Being of intrinsic interest for the theory of general branching processes, this information also sheds light on the organization of levels (the sets of vertices located at the same distance from the root) of some random trees (for instance, random recursive trees and binary search trees) that can be constructed as family trees of general branching processes stopped at suitable random times. We refer to \cite{Holmgren+Janson:2017} for more details and examples of the embeddable random trees.

\item[3)] Renewal theory for perturbed random walks is an inevitable ingredient for investigation of nested occupancy scheme in random environment generated by stick-breaking. Referring to  \cite{Buraczewski+Dovgay+Iksanov:2020, Iksanov+Marynych+Samoilenko:2020} for more details we only mention that the latter scheme is a generalization of the classical Karlin infinite balls-in-boxes occupancy scheme \cite{Gnedin+Hansen+Pitman:2007, Karlin:1967}. Unlike the Karlin scheme in which the collection of boxes is unique, there is a nested hierarchy of boxes, and the hitting probabilities of boxes are defined in terms of iterated stick-breaking. Assuming that $n$ balls have been thrown, denote by $K_n(j)$ the number of occupied boxes in the $j$th level which is the basic object of interest. It turns out that whenever $j=j_n=o((\log n)^{1/2})$ (the case of fixed $j$ is included) the distributional behavior of $K_n(j)$ as $n\to\infty$ is the same as that of $N_j(\log n)$, when the underlying perturbed random walk $T$ is appropriately chosen.

\end{enumerate}

We call the $j$th generation {\it early}, {\it intermediate} or {\it late} depending on whether $j$ is fixed, $j=j(t)\to\infty$ and $j(t)=o(t)$ as $t\to\infty$, or $j=j(t)$ is of order $t$. In view of Proposition \ref{bigg} given in Section \ref{sec:height} there are no other generations. Assume, for the time being, that $j$ is a late generation and that $T$ is a collection of random points, not necessarily the perturbed random walk. Nevertheless, we retain the notation $N_j$ and $V_j$. In this case the asymptotic behavior of $V_j$ and $N_j$ is well-understood. For instance, a delicate counterpart of the key renewal theorem for $V_j$ which includes both a version of the elementary renewal theorem and a version of Blackwell's theorem can be found in Theorem A of \cite{Biggins:1979}. For the corresponding a.s.\ result for $N_j$, see Theorem B of the same paper and Theorem 4 in \cite{Biggins:1992}. A strong law of large numbers for $N_j(bj)$ for appropriate $b>0$ is given in formula (1.1) of \cite{Biggins:1979}. From these and the other results of this flavor it follows that $N_j$ forgets what was happening in the early history and particularly in the $1$st generation. The behavior of $N_j$ is universal for a wide class of input processes (responsible for the $1$st generation). It is driven by limit theorems available for general branching processes like convergence of the Biggins martingales, large deviations etc.

While the present paper deals with some intermediate generations, the early generations which admit a much simpler analysis will be treated in a separate paper \cite{Iksanov etal:2021}. One may expect that the behavior of the iterated perturbed random walks in the early and intermediate generations is very different from that in the late generations. When $j$ is a non-late generation, the process $N_j$ should inherit, for the most part, the properties of $N$, possibly in a modified form. This statement is confirmed by counterparts of the elementary renewal theorem (Theorems \ref{thm:elem1} and \ref{thm:elem2}), the key renewal theorem (Theorem \ref{thm:keyren}) and Blackwell's theorem (Corollary \ref{black}) which are our main results.

The remainder of the paper is structured as follows. Our main findings are formulated in Section \ref{results} and then proved in Section \ref{proofs}. Also, Section \ref{results} contains two previously known results concerning $N_j$ and $V_j$. To our knowledge, all the results presented in this paper form the state-of-the-art as far as the intermediate generations of the iterated perturbed random walks are concerned. Finally, the appendix collects a rate of convergence result and counterparts of Blackwell's theorem and the key renewal theorem for the perturbed random walks.

\section{Results}\label{results}

\subsection{The height of a confined general branching process tree}\label{sec:height}

For $t>0$, put $$H(t):=\inf\{j\in\mn: N_j(t)=0\}$$ and note that $N_j(t)=0$ a.s.\ for all $j\geq H(t)$. We call the variable $H(t)$ the {\it height of a general branching process tree} generated by a perturbed random walk $T$ and confined to the strip $[0,t]$. The result given below is of principal importance for our classification of generations (early, intermediate, late).
\begin{assertion}\label{bigg}
For each $t\geq 0$, $H(t)<\infty$ a.s. Furthermore,
\begin{equation}\label{height}
\lim_{t\to\infty}\frac{H(t)}{t}=\frac{1}{\gamma}\in (0,\infty)\quad\text{{\rm a.s.}},
\end{equation}
where $\gamma:=\sup\{z>0: \mu(z)<1\}$ and $\mu(z):=\inf_{s>0}(e^{zs}\frac{\me e^{-s\eta}}{1-\me e^{-s\xi}})$ for $z>0$.
\end{assertion}
\begin{proof}
By assumption, $\mmp\{\eta=0\}=0$. This entails $\lim_{s\to\infty}\frac{\me e^{-s\eta}}{1-\me e^{-s\xi}}=0$ and thereupon
$$
\lim_{z\to 0+}\mu(z)=0.
$$
Also, $\lim_{z\to\infty}\mu(z)=\lim_{s\to 0+}\frac{\me e^{-s\eta}}{1-\me e^{-s\xi}}=\infty$. This shows that $\gamma\in (0,\infty)$.

Recall that, for $n\in\mn$, $(T_r^{(n)})_{r\in\mn}$ denotes some enumeration of the birth times in the $n$th generation of the general branching process. Put $B(n):=\inf_{r\geq 1} T_r^{(n)}$. By the famous Biggins result (Corollary on p.~635 in \cite{Biggins:1977}),
\begin{equation}\label{leftmost}
\lim_{n\to\infty}\frac{B(n)}{n}=\gamma\quad\text{a.s.}
\end{equation}
Since, for $n\in\mn$ and $t>0$, $\{H(t)>n\}=\{B(n)\leq t\}$ and, according to \eqref{leftmost}, $\lim_{n\to\infty}B(n)=+\infty$ a.s., we infer $H(t)<\infty$ a.s.

Finally, we have $B(H(t))>t\geq B(H(t)-1)$ a.s. The left-hand inequality ensures $\lim_{t\to\infty}H(t)=+\infty$ a.s.\ which together with \eqref{leftmost} proves \eqref{height} with the help of a standard sandwich argument.
\end{proof}

It is seldom possible to find the constant $\gamma$ explicitly. Here is one happy exception. Let $(\xi,\eta)=(|\log W|, |\log(1-W)|)$, where $W$ has a uniform distribution on $[0,\,1]$. The distribution of the sequence $(e^{-T_i})_{i\in\mn}$ is known as the {\it Griffiths--Engen--McCloskey distribution} with parameter $1$. In this case, $\mu(z)=ez$ for $z>0$ which gives $\gamma=e^{-1}$.

\subsection{Counterparts of the elementary renewal theorem for intermediate generations}\label{subsec:elem}

The simplest result of the renewal theory, called the elementary renewal theorem, tells us that
$$
U(t)=\sum_{i\geq 0}\mmp\{S_i\leq t\}~\sim~\frac{t}{{\tt m}},\quad t\to\infty,
$$
where ${\tt m}:=\me \xi<\infty$. Here and hereafter, the notation $f(t)\sim g(t)$ means that the ratio $f(t)/g(t)$ tends to $1$ as $t\to\infty$.

From \eqref{equ} it follows that, without any assumptions on $\eta$, 
\begin{equation}\label{eq:elem_v_1}
V(t)~\sim~\frac{t}{{\tt m}},\quad t\to\infty.
\end{equation}
This is a counterpart of the elementary renewal theorem for the perturbed random walks.

In this section we state two results on the first-order behavior of the convolutions powers $V_j$ of $V$. Our first result, Theorem \ref{thm:elem1}, deals with `early intermediate' generations satisfying $j=j(t)\to\infty$ and $j(t)=o(t^{1/2})$ as $t\to\infty$ as well as early generations. At this point we stress that even though both Theorem \ref{thm:elem1} and Theorem \ref{thm:elem2} hold true for early generations, the assumptions of these theorems are too restrictive as far as early generations are concerned. We refer to the forthcoming article \cite{Iksanov etal:2021} for a proper version of the elementary renewal theorem in early generations. Recall the standard notation $x\wedge y=\min (x,y)$ for $x,y\in\mr$.
\begin{thm}\label{thm:elem1}
Assume that either (i) $\E\xi^r<\infty$ for some $r\in (1,\,2]$ or (ii) $\mmp\{\xi>t\}~\sim~bt^{-r}$ for some $r\in (1,\,2)$ and some $b>0$. Suppose further that $\E(\eta\wedge t)=O(t^{2-r})$ as $t\to\infty$ with the same $r$ as in (i) or (ii). Then, for any integer-valued function $j=j(t)$ satisfying
$j(t)=o(t^{(r-1)/2})$ as $t\to\infty$,
\begin{equation}\label{eq:eq_elem1}
V_j(t)~\sim~\frac{t^j}{{\tt m}^j j!},\quad t\to\infty,
\end{equation}
where ${\tt m}=\me\xi<\infty$.
\end{thm}
\begin{rem}
The condition $\E\eta^{r-1}<\infty$ is sufficient for $\me(\eta\wedge t)=O(t^{2-r})$, $t\to\infty$. This follows from
\begin{multline*}
\me(\eta\wedge t)=\int_0^t \mmp\{\eta>y\}{\rm d}y\leq \int_0^t \Big(\frac{t}{y}\Big)^{2-r}\mmp\{\eta>y\}{\rm d}y\\
=t^{2-r}\int_0^\infty y^{r-2}\mmp\{\eta>y\}{\rm d}y=(r-1)^{-1}\me\eta^{r-1}t^{2-r}.
\end{multline*}
\end{rem}

Specializing Theorem \ref{thm:elem1} to $r=2$ gives the following corollary which has already been obtained
via a slightly different argument in formula (4.6) of \cite{Buraczewski+Dovgay+Iksanov:2020}.

\begin{cor}\label{cor:r=2}
Assume that $\me\xi^2<\infty$ and $\me \eta<\infty$. Then relation \eqref{eq:eq_elem1} holds for any integer-valued function $j=j(t)$ satisfying $j(t)=o(t^{1/2})$ as $t\to\infty$.
\end{cor}

Given next is a quite surprising result which shows that the convolution power $V_j$ exhibits a phase transition in the generations $j$ satisfying $j=j(t)\sim {\rm const}\,\cdot t^{1/2}$ as $t\to\infty$. Here, further moment and smoothness assumptions seem to be indispensable. In particular, we assume that the distribution of $\xi$ is spread-out, which means that some convolution power of the distribution function $t\mapsto \mmp\{\xi\leq t\}$ has an absolutely continuous component.

\begin{thm}\label{thm:elem2}
Assume that the distribution of $\xi$ is spread-out, 
that $\me \xi^3<\infty$ and $\me \eta^2<\infty$. Then,
for any integer-valued function $j=j(t)$ satisfying 
$j(t)=o(t^{2/3})$ as $t\to\infty$, 
$$
V_j(t)~\sim~ \frac{t^j}{{\tt m}^j j!}\exp{\left(\frac{\gamma_0{\tt m}j^2}{t}\right)},\quad t\to\infty,
$$
where
\begin{equation}\label{eq:gamma_0_def}
\gamma_0:=\int_{[0,\,\infty)}{\rm d}(V(y)-{\tt m}^{-1}y)=\lim_{t\to\infty}(V(t)-{\tt m}^{-1}t)=\frac{\me\xi^2}{2{\tt m}^2}-\frac{\me\eta}{{\tt m}}
\end{equation}
may be positive, negative or zero.
\end{thm}
\begin{rem}
Assume that $(\xi, \eta)=(|\log W|, |\log(1-W)|)$, where $W$ is a random variable having a uniform distribution on $[0,1]$. Then
$$
V_j(t)=\frac{t^j}{j!},\quad j\in\mn,\quad t\geq 0.
$$
This is in line with the asymptotics provided by Theorem \ref{thm:elem2}, for, in this case, $\gamma_0=0$ and ${\tt m}=1$.
\end{rem}

\subsection{Counterparts of the key renewal theorem and Blackwell's theorem for intermediate generations}

In the renewal theory the key renewal theorem is usually obtained as a corollary to Blackwell's theorem. We proceed differently by first proving a counterpart of the key renewal theorem (Theorem \ref{thm:keyren}) and then obtain a counterpart of Blackwell's theorem (Corollary \ref{black}) as a corollary. Recall that the distribution of $\xi$  is nonlattice if it is not concentrated on any centered lattice of the form $(dn)_{n\in\mn_0}$ for some $d>0$.
\begin{thm}\label{thm:keyren}
Let $f: [0,\infty)\to [0,\infty)$ be a directly Riemann integrable (dRi) function on $[0,\infty)$. Assume that either (a) or (b)
below holds true:
\begin{itemize}
\item[(a)] the distribution of $\xi$ is nonlattice, the conditions of Theorem \ref{thm:elem1} hold for some $r\in (1,\,2]$ (or some $r\in (1,\,2)$) and $j(t)=o(t^{(r-1)/2})$ as $t\to\infty$;
\item[(b)] the conditions of Theorem \ref{thm:elem2} hold and $j(t)=o(t^{2/3})$ as $t\to\infty$.
\end{itemize}
Then
\begin{equation}\label{eq:keyren}
(f\ast V_j)(t)=\int_{[0,\,t]}f(t-y){\rm d}V_j(y)~\sim~ \Big(\frac{1}{{\tt m}}\int_0^\infty f(y){\rm d}y\Big) V_{j-1}(t),\quad t\to\infty,
\end{equation}
where ${\tt m}=\me\xi<\infty$	, and $V_{j-1}(t)$ on the right-hand side can be replaced with $t^{j-1}/(m^{j-1} (j-1)!)$ in the case (a),
or with $t^{j-1}/(m^{j-1} (j-1)!)\exp{(\gamma_0{\tt m}j^2/t)}$ in the case (b).
\end{thm}

Upon taking $f(y)=\1_{[0,\,h)}(y)$ in Theorem \ref{thm:keyren} we immediately obtain the following. 

\begin{cor}\label{black}
Let $h>0$ be fixed. Under the assumptions of Theorem \ref{thm:keyren}
\begin{equation}\label{bla_early}
V_j(t+h)-V_j(t)~\sim~\frac{h}{m}V_{j-1}(t),\quad t\to\infty.
\end{equation}
\end{cor}

\subsection{A couple of previously known results}

In this section we collect two previously known facts concerning the asymptotic behaviour of $N_j$ in the intermediate generations. They are borrowed from \cite{Iksanov+Marynych+Samoilenko:2020} and stated here for integrity and the reader's convenience. We write ${\overset{{\rm f.d.d.}}\longrightarrow}$ to denote weak convergence of finite-dimensional distributions.

\begin{thm}[Multivariate central limit theorem for $(N_j(t))_{t\geq 0}$]\label{main5}
Assume that ${\tt s}^2={\rm Var}\,\xi\in (0,\infty)$ and $\me \eta<\infty$. Let $j=j(t)$ be any positive integer-valued function satisfying $j(t)\to \infty$ and $j(t)=o(t^{1/2})$ as $t\to\infty$. Then, as $t\to\infty$,
\begin{equation}\label{clt22}
\left(\frac{\lfloor j(t)\rfloor^{1/2}(\lfloor j(t)u\rfloor-1)!}{({\tt s}^2{\tt m}^{-2\lfloor j(t)u\rfloor-1}t^{2\lfloor j(t)u\rfloor-1})^{1/2}}\bigg(N_{\lfloor j(t)u\rfloor}(t)-V_{\lfloor j(t)u \rfloor}(t)\bigg)\right)_{u>0}\\
{\overset{{\rm f.d.d.}}\longrightarrow}~ \Bigg(\int_{[0,\,\infty)}e^{-uy}{\rm d}B(y)\Bigg)_{u>0},
\end{equation}
where $(B(v))_{v\geq 0}$ is a standard Brownian motion.
\end{thm}

According to Proposition 3.1, Theorems 3.2 and 3.3 in \cite{Buraczewski+Dovgay+Iksanov:2020}, the centering $V_{\lfloor j(t)u \rfloor}(t)$ in \eqref{clt22} can be replaced by its leading term
$$
t^{\lfloor j(t)u \rfloor}/((\lfloor j(t)u \rfloor)!{\tt m}^{\lfloor j(t)u \rfloor}),
$$
provided that $j(t)=o(t^{1/3})$. For functions $t\mapsto j(t)$ which grow faster, this is not always the case. Plainly, the possibility/impossibility of such a replacement is justified by a second-order behavior of $V_j$. It should come as no surprise that second-order results for $V_j$ require more restrictive assumptions on the distributions of $\xi$ and $\eta$ than the corresponding first-order results. The following proposition which is concerned with the rate of convergence in the elementary renewal theorem for $V_j$ was proved in Proposition 8.1 of \cite{Iksanov+Marynych+Samoilenko:2020}.
\begin{assertion}\label{prop:sec_order_behaviour}
Assume that the distribution of $\xi$ has an absolutely continuous component, that $\me e^{\beta_1\xi}<\infty$, $\me e^{\beta_2\eta}<\infty$ for some $\beta_1,\beta_2>0$ and
$$
\gamma_0=\frac{\me\xi^2}{2{\tt m}^2}-\frac{\me \eta}{{\tt m}}>0.
$$ Then
\begin{equation}\label{prec}
V_j(t)-\frac{t^j}{j!{\tt m}^j}~\sim~\frac{\gamma_0 jt^{j-1}}{(j-1)!{\tt m}^{j-1}},\quad t\to\infty
\end{equation}
whenever $j=j(t)=o(t^{1/2})$ as $t\to\infty$ ($j$ is allowed to be fixed).
\end{assertion}
Formula \eqref{prec} can be thought of as a generalization of formulae \eqref{eq:gamma_0_def} and \eqref{eq:sec_order_U}. These provide the second-order behaviour of the functions $V$ and $U$, respectively.

\section{Proofs}\label{proofs}

\subsection{Preparatory results}\label{Prep}

Recall that $U$ denotes the renewal function for $(S_n)_{n\in\mn_0}$. According to Lorden's inequality which holds whenever $\me\xi^2<\infty$,
\begin{equation}\label{lord}
U(t)-{\tt m}^{-1}t \leq c_0,\quad t\geq 0,
\end{equation}
where $c_0:=\me \xi^2/{\tt m}^2$ and ${\tt m}=\me\xi<\infty$. See \cite{Carlsson+Nerman:1986} for a nice proof under the assumption that the distribution of $\xi$ is nonlattice. Let $S_0^\ast$ be a random variable with distribution $\mmp\{S_0^\ast\in{\rm d}x\}={\tt m}^{-1}\mmp\{\xi>x\}\1_{(0,\infty)}(x){\rm d}x$. The basic formula (2) of the last cited paper which reads
\begin{equation}\label{eq:S0}
\me U(t-S_0^\ast)={\tt m}^{-1}t,\quad t\geq 0,
\end{equation}
holds true in the lattice case as well. Thus, the argument given in \cite{Carlsson+Nerman:1986} proves \eqref{lord} in general. We also note that
\begin{equation}\label{eq:sec_order_U}
\lim_{t\to\infty}\left(U(t)-{\tt m}^{-1}t\right)=\frac{\me\xi^2}{2{\tt m}^2}
\end{equation}
whenever the distribution of $\xi$ is nonlattice and $\me\xi^2<\infty$.

Since $V(t)\leq U(t)$ for $t\geq 0$ we infer
\begin{equation}\label{lord1}
V(t)-{\tt m}^{-1}t \leq c_0,\quad t\geq 0.
\end{equation}
On the other hand, assuming that $\me\eta<\infty$ (whereas the assumption $\me\xi^2<\infty$ is not needed here),
\begin{eqnarray*}
V(t)-{\tt m}^{-1}t&=&\int_{[0,\,t]}(U(t-y)-{\tt m}^{-1}(t-y)){\rm d}G(y)\\&\hphantom{==}-& {\tt m}^{-1} \int_0^t (1-G(y)){\rm d}y\geq-{\tt m}^{-1}\int_0^t (1-G(y)){\rm d}y\geq
-{\tt m}^{-1}\me\eta
\end{eqnarray*}
having utilized $U(t)\geq {\tt m}^{-1}t$ for $t\geq 0$ which is a consequence of Wald's identity $t\leq \E S_{\nu(t)}={\tt m} U(t)$, where $\nu(t):=\inf\{k\in\mn: S_k>t\}$ for $t\geq 0$. Thus, we have shown that,  under the assumptions $\me\xi^2<\infty$ and $\me\eta<\infty$,
\begin{equation}\label{lord2}
|V(t)-{\tt m}^{-1}t|\leq c_L,\quad t\geq 0
\end{equation}
where $c_L=\max(c_0, {\tt m}^{-1}\me\eta)$.

\subsection{Results on convolution powers of functions of a linear growth and proofs of Theorems \ref{thm:elem1} and \ref{thm:elem2}}

The results presented here are concerned with the following purely analytic problem. Assume that a nondecreasing function $f$ exhibits a linear growth, that is, $f(t)~\sim~at$ as $t\to\infty$ for some $a>0$. Then, for fixed $j\in\mn$,
$$
f^{\ast(j)}(t)~\sim~\frac{a^j t^j}{j!},\quad t\to\infty.
$$
Imposing various assumptions on the behavior of $f(t)-at$ we shall extend this asymptotics to the case when $j=j(t)$ diverges to infinity as $t\to\infty$.

\begin{assertion}\label{prop:convolutions1}
Let $f:\mr\to [0,\,\infty)$ be a nondecreasing right-continuous function vanishing on the negative half-line and satisfying
\begin{equation}\label{eq:asymp_exp_assump}
f(t)=at+O(t^{\alpha}),\quad t\to\infty
\end{equation}
for some $a>0$ and $\alpha\in [0,1)$.
Then, for any integer-valued function $j=j(t)$ such that
$j(t)=o(t^{(1-\alpha)/2})$ as $t\to\infty$,
$$
f_j(t):=f^{\ast(j)}(t)~\sim~\frac{a^j t^j}{j!},\quad t\to\infty.
$$
\end{assertion}
\begin{proof}
According to \eqref{eq:asymp_exp_assump} there exists $C\geq 1$ such that
\begin{equation}\label{eq:uniform_bound_f}
-C(t+1)^{\alpha} \leq f(t)-at \leq C(t+1)^{\alpha},\quad t\geq 0.
\end{equation}
For $j\in\mn$ and $t\geq 0$, put
$$
r_j(t):=\int_{[0,\,t]}f_j(t-y){\rm d}(f(y)-ay)=\int_{[0,\,t]}\left(f(t-y)-a(t-y)\right){\rm d}f_{j}(y)
$$
and note that
$$
f_j(t)=r_{j-1}(t)+a\int_0^t f_{j-1}(y){\rm d}y,\quad j\geq 2,~~ t\geq 0.
$$
By virtue of \eqref{eq:uniform_bound_f}, we conclude that
$$
|r_j(t)|\leq C(t+1)^{\alpha}f_j(t),\quad j\in\mn,~~ t\geq 0.
$$
Using this bound and the mathematical induction we obtain
\begin{equation}\label{eq:ineq}
W_j^{-}(t)\leq f_j(t)\leq W_j^{+}(t),\quad j\in\mn,~~t\geq 0,
\end{equation}
where $W_j^{\pm}$ is defined recursively by $W_0^{\pm}(t):=1$ and
$$
W_j^{\pm}(t)=\left(\pm C(t+1)^{\alpha}W^{\pm}_{j-1}(t)+a\int_{0}^t W_{j-1}^{\pm}(y){\rm d}y\right)_{+},\quad j\in\mn,~~ t\geq 0.
$$
Here, we recall that $x_{+}=\max(x,0)$ and note that taking the nonnegative part is only relevant for $W_j^{-}$ ensuring its nonnegativity, whereas it can be omitted for $W_j^+$.

It remains to show that
\begin{equation}\label{eq:asymp_for_bounds}
W_j^{\pm}(t)~\sim~\frac{a^j t^j}{j!},\quad t\to\infty.
\end{equation}
To this end, we first prove by induction that
\begin{equation}\label{eq:bound_w_j_plus}
W_j^{+}(t)\leq \frac{a^j t^j}{j!}+\sum_{i=0}^{j-1}\binom{j}{i}\frac{a^i C^{j-i}(t+1)^{\alpha(j-i)+i}}{i!},\quad j\in\mn,~~ t\geq 0.
\end{equation}
While for $j=1$ this follows immediately because $W_1^{+}(t)=C(t+1)^{\alpha}+at$,
the induction step works as follows
\begin{align*}
W_{j+1}^+(t)&\leq C(t+1)^{\alpha}\left(\frac{a^j t^j}{j!}+\sum_{i=0}^{j-1}\binom{j}{i}\frac{a^i C^{j-i}(t+1)^{\alpha(j-i)+i}}{i!}\right)\\
&+a\int_0^{t}\left(\frac{a^j y^j}{j!}+\sum_{i=0}^{j-1}\binom{j}{i}\frac{a^i C^{j-i}(y+1)^{\alpha(j-i)+i}}{i!}\right){\rm d}y\\
&\leq \sum_{i=0}^{j}\binom{j}{i}\frac{a^i C^{j+1-i}(t+1)^{\alpha(j+1-i)+i}}{i!}+\frac{a^{j+1}t^{j+1}}{(j+1)!}+\sum_{i=0}^{j-1}\binom{j}{i}\frac{a^{i+1}C^{j-i}}{i!}\frac{(t+1)^{\alpha(j-i)+i+1}}{\alpha(j-i)+i+1}\\
&\leq \frac{a^{j+1}t^{j+1}}{(j+1)!}+\sum_{i=0}^{j}\binom{j}{i}\frac{a^i C^{j+1-i}(t+1)^{\alpha(j+1-i)+i}}{i!}+\sum_{i=0}^{j-1}\binom{j}{i}\frac{a^{i+1}C^{j-i}}{(i+1)!}(t+1)^{\alpha(j-i)+i+1}\\
&= \frac{a^{j+1}t^{j+1}}{(j+1)!}+\sum_{i=0}^{j}\binom{j}{i}\frac{a^i C^{j+1-i}(t+1)^{\alpha(j+1-i)+i}}{i!}+\sum_{i=1}^{j}\binom{j}{i-1}\frac{a^iC^{j+1-i}}{i!}(t+1)^{\alpha(j+1-i)+i}\\
&= \frac{a^{j+1}t^{j+1}}{(j+1)!}+\sum_{i=0}^{j}\binom{j+1}{i}\frac{a^i C^{j+1-i}(t+1)^{\alpha(j+1-i)+i}}{i!}
\end{align*}
having utilized the binomial identity $\binom{j}{i}+\binom{j}{i-1}=\binom{j+1}{i}$ for the last step. Further,
\begin{align*}
&\hspace{-1cm}\frac{j!}{a^j t^j}\sum_{i=0}^{j-1}\binom{j}{i}\frac{a^i C^{j-i}(t+1)^{\alpha(j-i)+i}}{i!}\sim\frac{j!}{a^j (t+1)^j}\sum_{i=0}^{j-1}\binom{j}{i}\frac{a^i C^{j-i}(t+1)^{\alpha(j-i)+i}}{i!}\\
&\leq \sum_{i=0}^{j-1}\left(\frac{j!}{i!}\right)^{2}\left(\frac{C}{a}\right)^{j-i}(t+1)^{(1-\alpha)(i-j)} \leq \sum_{i=0}^{j-1}(j^{j-i})^2(Ca^{-1})^{j-i}(t+1)^{(1-\alpha)(i-j)}\\
&\leq \sum_{i\geq 1} \left(\frac{Ca^{-1} j^2}{(t+1)^{1-\alpha}}\right)^i=\frac{Ca^{-1} j^2}{(t+1)^{1-\alpha}}\left(1-\frac{Ca^{-1} j^2}{(t+1)^{1-\alpha}}\right)^{-1}.
\end{align*}
Thus, in view of the assumption $j(t)=o(t^{(1-\alpha)/2})$ we have
\begin{equation}\label{eq:asymp_for_bounds1}
\limsup_{t\to\infty}\frac{j!}{a^j t^j}W_j^{+}(t)\leq 1.
\end{equation}
To prove that
\begin{equation}\label{eq:bound_w_j_minus}
W_j^{-}(t)\geq \left(\frac{a^j t^j}{j!}-\sum_{i=0}^{j-1}\binom{j}{i}\frac{a^i C^{j-i}(t+1)^{\alpha(j-i)+i}}{i!}\right)_{+},\quad j\in\mn,~~ t\geq 0
\end{equation}
we use a similar reasoning. While \eqref{eq:bound_w_j_minus} is obviously true for $j=1$, we obtain with the help of induction, for $j\geq 2$,
\begin{align*}
&\hspace{-0.3cm}W_{j+1}^-(t)\geq -C(t+1)^{\alpha}W_{j}^{-}(t)+a\int_0^t W_{j}^{-}(y){\rm d}y\geq -C (t+1)^\alpha
W_{j}^{+}(t)+a\int_0^t W_{j}^{-}(y){\rm d}y\\
&\geq -C(t+1)^{\alpha}W_{j}^{+}(t)+a\int_0^t \left(\frac{a^j y^j}{j!}-\sum_{i=0}^{j-1}\binom{j}{i}\frac{a^iC^{j-i}(y+1)^{\alpha(j-i)+i}}{i!}\right){\rm d}y\\
&\geq -C(t+1)^{\alpha}\left(\frac{a^j t^j}{j!}+\sum_{i=0}^{j-1}\binom{j}{i}\frac{a^iC^{j-i}(t+1)^{\alpha(j-i)+i}}{i!}\right)\\
&+a\int_0^t \left(\frac{a^j y^j}{j!}-\sum_{i=0}^{j-1}\binom{j}{i}\frac{a^i C^{j-i}(y+1)^{\alpha(j-i)+i}}{i!}\right){\rm d}y\\
&\geq \frac{a^{j+1}t^{j+1}}{(j+1)!}-\left(\sum_{i=0}^{j}\binom{j}{i}\frac{a^i C^{j+1-i}(t+1)^{\alpha(j+1-i)+i}}{i!}+\sum_{i=0}^{j-1}\binom{j}{i}\frac{a^{i+1}C^{j-i}}{i!}\int_0^t (y+1)^{\alpha(j-i)+i}{\rm d}y\right)\\
&\geq \frac{a^{j+1}t^{j+1}}{(j+1)!}-\sum_{i=0}^{j}\binom{j+1}{i}\frac{a^i C^{j+1-i}(t+1)^{\alpha(j+1-i)+i}}{i!}.
\end{align*}
We have used \eqref{eq:ineq} and \eqref{eq:bound_w_j_plus} for the second and the fourth inequality, respectively. Since $W_{j+1}^-$ is nonnegative we arrive at \eqref{eq:bound_w_j_minus}. Thus,
\begin{equation}\label{eq:asymp_for_bounds2}
\liminf_{t\to\infty}\frac{j!}{a^j t^j}W_j^{-}(t)\geq 1.
\end{equation}
Combining \eqref{eq:asymp_for_bounds1} and \eqref{eq:asymp_for_bounds2} yields \eqref{eq:asymp_for_bounds}, thereby finishing the proof of Proposition \ref{prop:convolutions1}.
\end{proof}

\begin{proof}[Proof of Theorem \ref{thm:elem1}]
Theorem \ref{thm:elem1} is an immediate consequence of Proposition \ref{prop:convolutions1} and formula \eqref{Vsecond} of Lemma \ref{lem:secondorder} given in the appendix.
\end{proof}

The next results provides asymptotics of convolution powers $f^{\ast(j)}$ for $j=j(t)$ which may grow faster then $t^{1/2}$ under the assumption that the function $|f(t)-at|$ has a finite total variation and satisfies an additional integrability assumption. We shall use a convention that, for a function $x:\mr\to\mr$, $x^{\ast(0)}(t)=\1_{[0,\infty)}(t)$, $t\in\mr$. Also, we shall write $\mathcal{V}_I(x)$ for the total variation of $x$ on the (possibly infinite) interval $I$. Finally, if $x$ is a function of a finite total variation on $[a,\,b]$, $-\infty\leq a<b\leq\infty$ and $y$ is a measurable function on $I$, we stipulate that
$$
\int_{[a,\,b]} y(t)|{\rm d}x(t)|=\int_{[a,\,b]} y(t){\rm d}\left(\mathcal{V}_{[a,\,t]}(x)\right),
$$
where the integral on the right-hand side is understood in the Lebesgue--Stieltjes sense.
\begin{assertion}\label{prop:convolutions2}
Let $f:\mr\mapsto [0,\,\infty)$ be a nondecreasing right-continuous function vanishing on the negative half-line. Assume that the function $\varepsilon$ defined by
\begin{equation}\label{eq:asymp_exp_assump2}
\varepsilon(t):=f(t)-at,\quad t\geq 0,
\end{equation}
for some $a>0$, satisfies
\begin{equation}\label{eq:asymp_inegrability1}
\int_{[0,\,\infty)}y|{\rm d}\varepsilon(y)|<\infty.
\end{equation}
Then, for any integer-valued function $j=j(t)$ such that $j(t)=o(t^{2/3})$ as $t\to\infty$, 
\begin{equation}\label{eq:asymp_exp_res2}
f_j(t):=f^{\ast(j)}(t)~\sim~\frac{a^j t^j}{j!}\exp{\left(\frac{\gamma_0 j^2}{a t}\right)},\quad t\to\infty,
\end{equation}
where $\gamma_0:=\int_{[0,\,\infty)}{\rm d}\varepsilon(y)=\lim_{t\to\infty}(f(t)-at)$.
\end{assertion}
\begin{proof}
The function $\varepsilon$, as the difference of two nondecreasing functions, has a finite total variation on every finite interval. In particular, \eqref{eq:asymp_inegrability1} entails
$$
\int_{[0,\,\infty)}|{\rm d}\varepsilon(y)|\leq \int_{[0,\,1)}|{\rm d}\varepsilon(y)|+\int_{[1,\,\infty)}y|{\rm d}\varepsilon(y)|<\infty.
$$
Thus, $\varepsilon$ has a finite total variation on $[0,\infty)$. Write
$$
\int_0^{\infty}|\varepsilon(y)-\gamma_0|{\rm d}y=\int_0^{\infty}\left|\int_{(y,\,\infty)}{\rm d}\varepsilon(z)\right|{\rm d}y\leq \int_0^{\infty}\int_{(y,\,\infty)}\left|{\rm d}\varepsilon(z)\right|{\rm d}y=\int_{[0,\,\infty)}y\left|{\rm d}\varepsilon(y)\right|<\infty
$$
having utilized integration by parts for the last equality. Hence, \eqref{eq:asymp_inegrability1} implies that
\begin{equation}\label{eq:asymp_inegrability2}
\int_0^\infty |\varepsilon(y)-\gamma_0|{\rm d}y<\infty.
\end{equation}
Now we modify \eqref{eq:asymp_exp_assump2} in a neighborhood of the origin, so that the
essential properties of $\varepsilon$ given by  \eqref{eq:asymp_inegrability1} and \eqref{eq:asymp_inegrability2} are preserved. Put
\begin{equation}\label{eq:decomposition_modified}
f(t)=(at+\gamma_0)_++\widetilde{\varepsilon}(t)=:\ell(t)+\widetilde{\varepsilon}(t),\quad t\in\mathbb{R}.
\end{equation}
Note that both summands 
can be non-zero in a bounded left neighborhood of the origin, yet
\begin{equation}\label{eq:asymp_inegrability3}
\int_{\mathbb{R}}|\widetilde{\varepsilon}(y)|{\rm d}y<\infty\quad\text{and}\quad \int_{\mathbb{R}}|y| |{\rm d}\widetilde{\varepsilon}(y)|<\infty
\end{equation}
because $t\mapsto \varepsilon(t)-\gamma_0-\widetilde{\varepsilon}(t)$  has a bounded support. The advantage of \eqref{eq:decomposition_modified} is justified by a 
simple formula for the convolution powers of $\ell$, namely,
$$
\ell^{\ast(j)}(t)=\frac{(at+\gamma_0 j)^j_{+}}{j!},\quad j\in\mn,~~ t\in\mathbb{R}.
$$
To check this we use the mathematical induction. While, for $j=1$, the formula is trivial,
the induction step works as follows: for $t\geq -a^{-1}\gamma_0(j+1)$,
\begin{align*}
\ell^{\ast(j+1)}(t)=\int_{\mr}\frac{(a(t-y)+\gamma_0 j)^j_{+}}{j!}{\rm d}\ell(y)=a\int_{-\gamma_0 a^{-1}}^{t+j\gamma_0 a^{-1}}\frac{(a(t-y)+\gamma_0 j)^j}{j!}{\rm d}y\\
=\int_0^{at+\gamma_0(j+1)}\frac{z^j}{j!}{\rm d}z=\frac{(at+\gamma_0(j+1))^{j+1}}{(j+1)!},
\end{align*}
and $\ell^{\ast(j+1)}(t)=0$ for $t<-a^{-1}\gamma_0(j+1)$.

We intend to prove \eqref{eq:asymp_exp_res2}. Using \eqref{eq:decomposition_modified} we obtain
$$
f^{\ast(j)}(t)=\ell^{\ast(j)}(t)+\sum_{k=0}^{j-1}\binom{j}{k}\left(\ell^{\ast(k)}\ast\widetilde{\varepsilon}^{\ast(j-k)}\right)(t),\quad t\in \mr.
$$
We are going to show that the second summand is asymptotically negligible with respect to
$\ell^{\ast(j)}(t)$ whenever $j(t)=o(t^{2/3})$. Assume this has already been done. Then
\eqref{eq:asymp_exp_res2} follows immediately because, for large enough $t$,
$$
f^{\ast(j)}(t)~=~\ell^{\ast(j)}(t)~=~\frac{a^jt^j}{j!}\left(1+\frac{\gamma_0 j}{at}\right)^j~=~\frac{a^jt^j}{j!}\exp{\left(j\log \left(1+\frac{\gamma_0j}{at}\right)\right)}.
$$
The right-hand side is asymptotically equivalent to $\frac{a^jt^j}{j!}\exp{\left(\frac{\gamma_0 j^2}{at}\right)}$ whenever $j=j(t)=o(t^{2/3})$ as $t\to\infty$.

Passing to the analysis of
$$
R_j(t):=\sum_{k=0}^{j-1}\binom{j}{k}\left(\ell^{\ast(k)}\ast\widetilde{\varepsilon}^{\ast(j-k)}\right)(t),\quad t\geq 0
$$
we first check that
\begin{equation}\label{eq:finite_variation}
\mathcal{V}_{\mr}(\ell\ast \widetilde{\varepsilon})\leq \widetilde{C}<\infty
\end{equation}
for an absolute constant $C>0$. For $t\in\mr$,
$$
(\ell\ast\widetilde{\varepsilon})(t)=\int_{\mr}\widetilde{\varepsilon}(t-y){\rm d}\ell(y)=a\int_{-a^{-1}\gamma_0}^\infty \widetilde{\varepsilon}(t-y){\rm d}y=a\int_{-\infty}^{t+a^{-1}\gamma_0}\widetilde{\varepsilon}(y){\rm d}y.
$$
Thus, \eqref{eq:finite_variation} holds with $\widetilde{C}:=a\int_{\mr}|\widetilde{\varepsilon}(y)|{\rm d}y$. Put
$$
g_{i,j}(t):=\mathcal{V}_{(-\infty,\,t]}(\ell^{\ast(i)} \ast \widetilde{\varepsilon}^{\ast(j)}),\quad i,j\in\mn_0,\quad t\in\mr.
$$
Then, for $i,j\in\mn$,
\begin{multline*}
g_{i,j}(t)=\mathcal{V}_{(-\infty,\,t]}\left((\ell^{\ast(i-1)} \ast \widetilde{\varepsilon}^{\ast(j-1)})\ast (\ell\ast \widetilde{\varepsilon})\right)\leq \mathcal{V}_{(-\infty,\,t]}(\ell^{\ast(i-1)} \ast \widetilde{\varepsilon}^{\ast(j-1)})\mathcal{V}_{(-\infty,\,t]} (\ell\ast \widetilde{\varepsilon})\\
\leq \mathcal{V}_{(-\infty,\,t]}(\ell^{\ast(i-1)} \ast \widetilde{\varepsilon}^{\ast(j-1)})\mathcal{V}_{\mr} (\ell\ast \widetilde{\varepsilon})\leq \widetilde{C}g_{i-1,j-1}(t),\quad t\in\mr,
\end{multline*}
where we have used that the total variation of the convolution of two functions is bounded by the product of their total variations, see Theorem 1.3.2(c) in \cite{Rudin:1962}. Iterating this inequality we conclude that
$$
|R_j(t)|\leq \sum_{k=0}^{j-1}\binom{j}{k}g_{k,j-k}(t)\leq \sum_{k\leq j/2}\binom{j}{k}\widetilde{C}^k g_{0,j-2k}(t)+\sum_{j/2<k<j}\binom{j}{k}\widetilde{C}^{j-k} g_{2k-j,0}(t),\quad t\in\mr.
$$
Note that $g_{0,j-2k}(t)\leq \mathcal{V}_{\mr}(\widetilde{\varepsilon}^{\ast(j-2k)})\leq (\mathcal{V}_{\mr}(\widetilde{\varepsilon}))^{j-2k}\leq \widetilde{C}_1^{j-2k}$ for $\widetilde{C}_1:=\int_{\mr}|{\rm d}\widetilde{\varepsilon}(y)|<\infty$. Therefore, 
$$
\sum_{k\leq j/2}\binom{j}{k}\widetilde{C}^k g_{0,j-2k}(t)\leq \sum_{k\leq j/2}\binom{j}{k}\widetilde{C}^k\widetilde{C}^{j-2k}_1\leq (\widetilde{C}\widetilde{C}_1^{-1}+\widetilde{C}_1)^j=o\left(\frac{a^jt^j}{j!}\left(1+\frac{\gamma_0 j}{at}\right)^j\right),\quad t\to\infty,
$$
for $b^j=b^{j(t)}$ grows slower than $\frac{a^jt^j}{j!}\left(1+\frac{\gamma_0 j}{at}\right)^j$ as $t\to\infty$ for an arbitrary finite constant $b>0$. Now we analyze the second sum
\begin{multline*}
\sum_{j/2<k<j}\binom{j}{k}\widetilde{C}^{j-k} g_{2k-j,0}(t)=\sum_{j/2<k<j}\binom{j}{k}\widetilde{C}^{j-k}\mathcal{V}_{(-\infty,\,t]}(\ell^{\ast(2k-j)})=\sum_{j/2<k<j}\binom{j}{k}\widetilde{C}^{j-k}\ell^{\ast(2k-j)}(t)\\
=\sum_{j/2<k<j}\binom{j}{k}\widetilde{C}^{j-k}\frac{(at +\gamma_0(2k-j))^{2k-j}}{(2k-j)!}=\sum_{1\leq k<j/2}\binom{j}{k}\widetilde{C}^{k}\frac{(at +\gamma_0(j-2k))^{j-2k}}{(j-2k)!}.
\end{multline*}
Here, the second equality follows from monotonicity of $\ell$ and the third equality holds for $t$ large enough. It is important for what follows that, for $k<j/2$ and $t>0$,
$$
\frac{(at +\gamma_0(j-2k))^{j-2k}}{(j-2k)!}\leq \frac{a^{j-2k}t^{j-2k}}{(j-2k)!}\exp{\left(\frac{\gamma_0 (j-2k)^2}{at}\right)}.
$$

\noindent
{\sc Case $\gamma_0\geq 0$.} We obtain, for $t>0$,
\begin{align*}
&\hspace{-0.3cm}\sum_{1\leq k<j/2}\binom{j}{k}\widetilde{C}^{k}\frac{(at +\gamma_0(j-2k))^{j-2k}}{(j-2k)!}\leq 
\exp{\left(\frac{\gamma_0 j^2}{at}\right)}\sum_{1\leq k<j/2}\binom{j}{k}\widetilde{C}^{k}\frac{a^{j-2k}t^{j-2k}}{(j-2k)!}\\
&= 
\frac{a^jt^j}{j!}\exp{\left(\frac{\gamma_0 j^2}{at}\right)}\sum_{1\leq k<j/2}\frac{(j!)^2}{(j-k)!(j-2k)!}\frac{1}{k!}\frac{\widetilde{C}^{k}}{a^{2k}t^{2k}}\leq 
\frac{a^jt^j}{j!}\exp{\left(\frac{\gamma_0 j^2}{at}\right)}\sum_{k\geq 1}j^{3k}\frac{1}{k!}\frac{\widetilde{C}^{k}}{a^{2k}t^{2k}}\\
&=\frac{a^jt^j}{j!}\exp{\left(\frac{\gamma_0 j^2}{at}\right)}\left(\exp{\left(\frac{\widetilde{C}j^3}{a^2t^2}\right)}-1\right).
\end{align*}
The last factor converges to zero whenever $j=j(t)=o(t^{2/3})$, whence the claim.

\noindent
{\sc Case $\gamma_0<0$.} Arguing in the same vein it is enough to check that
$$
\sum_{1\leq k<j/2}\frac{1}{k!}\left(\frac{\widetilde{C}j^3}{a^2t^2}\right)^k\exp{\left(\frac{\gamma_0(j-2k)^2}{at}\right)}=o\left(\exp{\left(\frac{\gamma_0j^2}{at}\right)}\right),\quad t\to\infty
$$
which is equivalent to
$$
I_t:=\sum_{1\leq k<j/2}\frac{1}{k!}\left(\frac{\widetilde{C}j^3}{a^2t^2}\right)^k\exp{\left(\frac{4|\gamma_0|k(j-k)}{at}\right)}=o(1),\quad t\to\infty.
$$
Invoking the inequality $\exp{\left(\frac{4|\gamma_0|k(j-k)}{at}\right)}\leq \exp{(4|\gamma_0|a^{-1}k)}$ for $1\leq k<j$ and large enough $t$ we infer $$I_t\leq \sum_{1\leq k<j/2}\frac{1}{k!}\left(\frac{\widetilde{C}j^3\exp(4|\gamma_0|a^{-1})}{a^2t^2}\right)^k\leq\exp\left(\frac{\widetilde{C}j^3}{a^2t^2}\exp{(4|\gamma_0|a^{-1})}\right)-1~\to~0,\quad t\to\infty.$$
The proof of Proposition \ref{prop:convolutions2} is complete.
\end{proof}

\begin{proof}[Proof of Theorem \ref{thm:elem2}]
We intend to apply Proposition \ref{prop:convolutions2}. To this end, it is enough to check that, under the assumptions of Theorem \ref{thm:elem2},
$$\int_{[0,\,\infty)}y|{\rm d}(V(y)-{\tt m}^{-1}y)|<\infty.$$
Recall that $V=U\ast G$ and denote by ${\rm Id}$ the identity function on $[0,\infty)$, that is, ${\rm Id}(t):=t_{+}=t\1_{[0,\infty)}(t)$ for $t\in\mr$. Then
$$
V-{\tt m}^{-1}{\rm Id}=(U-{\tt m}^{-1}{\rm Id})\ast G-{\tt m}^{-1}({\rm Id}\ast (1-G)).
$$
Using this and integration by parts yields
\begin{align*}
&\int_{[0,\,\infty)}y|{\rm d}(V(y)-{\tt m}^{-1}y)|=-\int_{[0,\,\infty)}y{\rm d}\mathcal{V}_{[y,\,\infty)}(V-{\tt m}^{-1}{\rm Id})=\int_{[0,\,\infty)}\mathcal{V}_{[y,\,\infty)}(V-{\tt m}^{-1}{\rm Id}){\rm d}y\\
&\hspace{2cm}\leq \int_{[0,\,\infty)}\mathcal{V}_{[y,\,\infty)}(U-{\tt m}^{-1}{\rm Id}){\rm d}y+{\tt m}^{-1}\int_{[0,\,\infty)}\mathcal{V}_{[y,\,\infty)}({\rm Id}\ast (1-G)){\rm d}y\\
&\hspace{2cm}=\int_{[0,\,\infty)}y|{\rm d}(U(y)-{\tt m}^{-1}y)|+{\tt m}^{-1}\int_0^{\infty}\int_y^{\infty} (1-G(z)){\rm d}z{\rm d}y.
\end{align*}
The first summand is finite by Remark 3.1.7(ii) on p.~121 in \cite{Frenk:1982} and the second is finite in view of
the assumption $\me\eta^2<\infty$.

The explicit form of $\gamma_0$ follows from the decomposition
$$
V(t)-{\tt m}^{-1}t=\int_{[0,\,t]}(U(t-y)-{\tt m}^{-1}(t-y)){\rm d}G(y)-{\tt m}^{-1}\int_{[0,\,t]}y{\rm d}G(y)-{\tt m}^{-1}t(1-G(t)),
$$
in which the first summand converges to $(2{\tt m}^2)^{-1}\me\xi^2$ by the dominated convergence theorem, \eqref{lord} and \eqref{eq:sec_order_U}; the second converges to $-{\tt m}^{-1}\me\eta$ and the third tends to zero as $t\to\infty$.
\end{proof}

Finally, we give a general result on the behavior of $f^{\ast(j)}$ for arbitrary $j=j(t)=o(t)$. Unfortunately, this result can seldom be applied to the counting function $V$ but is of independent interest and has at least two merits. On the one hand, it gives a probabilistic explanation of a rather mysterious appearance of the exponent in \eqref{eq:asymp_exp_res2}. On the other hand, it may be used for guessing the behaviour of $V_j$ for $j=j(t)$ growing at least as fast as $t^{2/3}$.
\begin{assertion}\label{prop:convolutions3}
Let $(\widetilde{S}_j)_{j\in\mn_0}$ be a nondecreasing zero-delayed standard random walk with $K(t):=\mmp\{\widetilde{S}_1\leq t\}$ for $t\in\mr$. Assume that, for some $a>0$,
$$
f(t)=at-\int_0^t(1-K(y)){\rm d}y,\quad t\geq 0.
$$
Then
$$
f^{\ast(j)}(t)=\frac{\me (at-\widetilde{S}_j)_{+}^j}{j!},\quad j\in\mn,~~ t\geq 0.
$$
In particular, if $\me \widetilde{S}_1^2<\infty$ and $j=j(t)=o(t^{2/3})$ as $t\to\infty$, then \eqref{eq:asymp_exp_res2} holds with $\gamma_0=-\me \widetilde{S}_1$.
\end{assertion}
\begin{proof}
Replacing $K$ with $t\mapsto K(at)$ we can and do assume that $a=1$, that is, $f(t)=\int_0^t K(y){\rm d}y$ or, in short, $f=K\ast {\rm Id}$. Then
$$
f^{\ast(j)}(t)=\left(({\rm Id})^{\ast(j)}\ast K^{\ast(j)}\right)(t)=\int_{[0,\,t]}\frac{(t-y)^j}{j!}{\rm d}K^{\ast(j)}(y)=\frac{\me (t-\widetilde{S}_j)_{+}^j}{j!},\quad t\geq 0.
$$
If $\me \widetilde{S}_1^2<\infty$, then $j=j(t)=o(t^{2/3})$ as $t\to\infty$ implies that
\begin{equation}\label{eq:derivation_rw_asymptotics}
\me (t-\widetilde{S}_j)_{+}^j~\sim~t^j\exp\left(\frac{\gamma_0j^2}{t}\right),\quad t\to\infty.
\end{equation}
This can be justified as follows. We first note that $\gamma_0<0$. Further, in the decomposition
\begin{equation}\label{eq:decomposition_gen_conv}
\me\left(1-\frac{\widetilde{S}_j}{t}\right)_{+}^j=\me (e^{j\log(1-\widetilde{S}_j/t)}\1_{\{\widetilde{S}_j\leq t/2\}})+\me\left(1-\frac{\widetilde{S}_j}{t}\right)_{+}^j\1_{\{\widetilde{S}_j\in (t/2,\,t)\}}
\end{equation}
the second summand is bounded by $2^{-j}$ and $2^{-j}=o\left(\exp\left(\frac{\gamma_0j^2}{t}\right)\right)$ as $t\to\infty$, for $j^2/t=o(j)$. The first summand in \eqref{eq:decomposition_gen_conv} can be bounded with the help of the inequalities
$$
-x-x^2 \leq \log(1-x)\leq -x,\quad x\in [0,\,1/2]\quad\text{and}\quad 1-x\leq e^{-x},\quad x\in\mathbb{R}.
$$
Indeed, we obtain, for $j\geq 4$,
\begin{multline*}
\me e^{-j\widetilde{S}_j/t}\left(1-\frac{j\widetilde{S}_j^2}{t^2}\right)\leq \me e^{-j\widetilde{S}_j/t}\left(1-\frac{j\widetilde{S}_j^2}{t^2}\right)\1_{\{\widetilde{S}_j\leq t/2\}}\leq \me e^{-j\widetilde{S}_j/t}e^{-j\widetilde{S}_j^2/t^2}\1_{\{\widetilde{S}_j\leq t/2\}}\\
\leq \me (e^{j\log(1-\widetilde{S}_j/t)}\1_{\{\widetilde{S}_j\leq t/2\}})\leq \me e^{-j\widetilde{S}_j/t}\1_{\{\widetilde{S}_j\leq t/2\}}\leq \me e^{-j\widetilde{S}_j/t}.
\end{multline*}
For $\lambda\geq 0$, put $\phi(\lambda):=\me e^{-\lambda\widetilde{S}_1}$. In view of $\me\widetilde{S}_1^2<\infty$ we infer 
$$
\me e^{-j\widetilde{S}_j/t}=\phi^j(j/t)=\left(1+\frac{\gamma_0 j}{t}+O\left(\frac{j^2}{t^2}\right)\right)^j.
$$
The right-hand side is asymptotically equivalent to $\exp(\gamma_0 j^2/t)$ as $t\to\infty$ under the assumption $j=j(t)=o(t^{2/3})$. Finally, the relation 
$$
\me e^{-j\widetilde{S}_j/t}\left(\frac{j\widetilde{S}_j^2}{t^2}\right)=o\left(\me e^{-j\widetilde{S}_j/t}\right),\quad t\to\infty
$$
can be checked using the equality $\me e^{-j\widetilde{S}_j/t}\widetilde{S}_j^2=\frac{\partial^2}{\partial \lambda^2}(\phi^j (\lambda))\Big|_{\lambda=j/t}$ in conjunction with the assumptions $j=j(t)=o(t^{2/3})$ and $\me\widetilde{S}_1^2<\infty$.
\end{proof}

\begin{rem}\label{rem:remove}
In the setting of Proposition \ref{prop:convolutions3}, assume that $\me \widetilde{S}_1^3<\infty$ and $j=j(t)=o(t^{3/4})$ as $t\to\infty$. We state, without going into details (which become rather technical), that
$$\me\big(at-\widetilde{S}_j\big)_+^j~\sim~a^jt^j\exp{\left(\gamma_0 j^2/t+(\gamma_1/2-\gamma_0^2)j^3/t^2\right)},\quad t\to\infty,
$$
where $\gamma_0=-\me \widetilde{S}_1$ and $\gamma_1:=\me \widetilde{S}_1^2$. 
\end{rem}

\subsection{Proof of Theorem \ref{thm:keyren} }

For $t\geq 0$, put $g(t):=\int_{[0,\,t]}f(t-y){\rm d}V(y)$ and $I:={\tt m}^{-1}\int_0^\infty f(y){\rm d}y$. By Lemma \ref{key}(a), given $\varepsilon>0$ there exists $t_0>0$ such that $|g(t)-I|\leq \varepsilon$ whenever $t\geq t_0$. Also, by Lemma \ref{keylight}, $g(t)\leq J$ for some $J>0$ and all $t\geq 0$. Hence, for $t\geq t_0$,
\begin{multline}\label{eq:keyren_eq1}
(f\ast V_j)(t)=(g\ast V_{j-1})(t)=\int_{[0,\,t]}g(t-y){\rm d}V_{j-1}(y)=\int_{[0,\,t-t_0]}g(t-y){\rm d}V_{j-1}(y)\\
+\int_{(t-t_0,\,t]}g(t-y){\rm d}V_{j-1}(y) \leq (I+\varepsilon)V_{j-1}(t)+J (V_{j-1}(t)-V_{j-1}(t-t_0)).
\end{multline}
We claim that
\begin{equation}\label{eq:keyren_increment}
\lim_{t\to\infty}\frac{V_{j(t)-1}(t)-V_{j(t)-1}(t-t_0)}{V_{j(t)-1}(t)}=0.
\end{equation}
Note that \eqref{eq:keyren_increment} is not a direct consequence of the elementary renewal theorem, for the theorem provides the asymptotics of $V_{j(t-t_0)-1}(t-t_0)$ rather than $V_{j(t)-1}(t-t_0)$ which is actually needed for \eqref{eq:keyren_increment}. To prove \eqref{eq:keyren_increment} we write with the help of \eqref{subad}
$$
0\leq V_{j(t)-1}(t)-V_{j(t)-1}(t-t_0)=\int_{[0,\,t]}(V(t-y)-V(t-t_0-y)){\rm d}V_{j(t)-2}(y)\leq U(t_0)V_{j(t)-2}(t),
$$
for all $t\geq 0$. Thus, \eqref{eq:keyren_increment} follows from
$$
\lim_{t\to\infty}\frac{V_{j(t)-2}(t)}{V_{j(t)-1}(t)}=0,
$$
which is a consequence of Theorems \ref{thm:elem1} and \ref{thm:elem2} applied with $j=j(t)-1$ and $j=j(t)-2$.

Combining \eqref{eq:keyren_eq1} and \eqref{eq:keyren_increment} we obtain
$$
\limsup_{t\to\infty}\frac{(f\ast V_j)(t)}{V_{j-1}(t)}\leq I. 
$$
The converse inequality for the limit inferior follows analogously. The remaining statements of Theorem \ref{thm:keyren} are secured by Theorems \ref{thm:elem1} and \ref{thm:elem2}.

\section{Appendix}

In this section we shall prove counterparts for perturbed random walks of some standard renewal-theoretic results. Recall that, under the sole assumption ${\tt m}=\me\xi<\infty$,
$$\lim_{t\to\infty}\frac{U(t)}{t}=\lim_{t\to\infty}\frac{V(t)}{t}=\frac{1}{{\tt m}}.$$ We start by discussing the rate of convergence in both limit relations.
\begin{lemma}\label{lem:secondorder}
Assume that either (i) $\E\xi^r<\infty$ for some $r\in (1,\,2]$, or (ii) $\mmp\{\xi>t\}~\sim~bt^{-r}$ for some $r\in (1,\,2)$ and some $b>0$ as $t\to\infty$. Then
\begin{equation}\label{eq:u_two_terms_exp}
U(t)=\frac{t}{{\tt m}}+O(t^{2-r}),\quad t\to\infty,
\end{equation}
where ${\tt m}=\me\xi<\infty$. In the case (i) when $r\in (1,2)$, the big $O$ can be replaced with a little $o$.

Under the additional assumption $\E(\eta\wedge t)=O(t^{2-r})$ as $t\to\infty$,
\begin{equation}\label{Vsecond}
V(t)=\frac{t}{{\tt m}}+O(t^{2-r}),\quad t\to\infty.
\end{equation}
\end{lemma}
\begin{proof}
First, we focus on \eqref{eq:u_two_terms_exp}.

\noindent
{\sc Case (i)}. If $r=2$, then \eqref{eq:u_two_terms_exp} follows from Lorden's inequality \eqref{lord}. Assume now that $r\in (1,2)$ and note that, for any $p>1$, $\me \xi^p<\infty$ is equivalent to $\me (S_0^\ast)^{p-1}<\infty$. The situation is not excluded that $\me S_0^\ast<\infty$ in which case $\me\xi^2<\infty$, so that $U(t)={\tt m}^{-1}t +O(1)={\tt m}^{-1}t+O(t^{2-r})$ as $t\to\infty$. Thus, in what follows we can and do assume that $\me S_0^\ast=\infty$. Then $$U(t)-{\tt m}^{-1}t=\int_{[0,\,t]}\mmp\{S_0^\ast>t-y\}{\rm d}U(y)~\sim~{\tt m}^{-1}\int_0^t \mmp\{S_0^\ast>y\}{\rm d}y,\quad t\to\infty,$$ where the equality is nothing else but \eqref{eq:S0}, and the asymptotic relation follows from Theorem 4 in \cite{Sgibnev:1981}. Now $\me (S_0^\ast)^{r-1}<\infty$ entails $\mmp\{S_0^\ast>t\}=o(t^{1-r})$, whence $\int_0^t \mmp\{S_0^\ast>y\}{\rm d}y=o(t^{2-r})$ as $t\to\infty$.

\noindent {\sc Case (ii)}. In this case $\mmp\{S_0^\ast>t\}\sim b({\tt m}(r-1))^{-1} t^{-(r-1)}$ as $t\to\infty$. This implies that
\begin{equation*}
U(t)-\frac{t}{\tt m}~\sim~\frac{1}{{\tt m}}\int_0^t \mmp\{S_0^\ast>y\}{\rm d}y~\sim~\frac{b}{{\tt m^2}(r-1)(2-r)}t^{2-r},\quad t\to\infty
\end{equation*}
and thereupon \eqref{eq:u_two_terms_exp}. Under the additional assumption that the distribution of $\xi$ is nonlattice relation \eqref{eq:u_two_terms_exp} also follows from Theorem 2.2 in \cite{Mohan:1976}.

Finally, relation \eqref{Vsecond} follows from the equality (which has already appeared in Section \ref{Prep})
$$V(t)-{\tt m}^{-1}t=\int_{[0,\,t]}(U(t-y)-{\tt m}^{-1}(t-y)){\rm d}G(y)-{\tt m}^{-1}\me (\eta\wedge t)$$
because each summand is $O(t^{2-r})$ by \eqref{eq:u_two_terms_exp} and the assumption of the theorem, respectively.
\end{proof}

We continue by noting that
\begin{equation}\label{subad}
V(x+y)-V(x)\leq U(y),\quad x,y\in\mr.
\end{equation}
Indeed, for $x,y\geq 0$,
\begin{eqnarray}\label{eqV}
V(x+y)-V(x)&=&\me (U(x+y-\eta)-U(x-\eta))\1_{\{\eta\leq x\}}+\me U(x+y-\eta)\1_{\{x<\eta\leq x+y\}}\\&\leq& U(y)(\mmp\{\eta\leq x\}+\mmp\{x<\eta\leq x+y\})\leq U(y)\notag
\end{eqnarray}
having utilized subadditivity and monotonicity of $U$ for the penultimate inequality. If $x,y<0$, then both sides of \eqref{subad} are zero. Finally, we use monotonicity of $V$ to obtain: if $x<0$ and $y\geq 0$, then $V(x+y)-V(x)=V(x+y)\leq V(y)\leq U(y)$; and if $x\geq 0$ and $y<0$, then $V(x+y)-V(x)\leq 0=U(y)$.

Lemmas \ref{bla} and \ref{key} are counterparts of Blackwell's theorem and the key renewal theorem, respectively. Observe that the presence of the $\eta_k$ plays no role, and the results are of the same form as for renewal functions.

\begin{lemma}\label{bla}
Let $h>0$ be any fixed number.

\noindent (a) Assume that the distribution of $\xi$ is nonlattice and ${\tt m}=\me\xi<\infty$. Then $$\lim_{t\to\infty} (V(t+h)-V(t))={\tt m}^{-1}h.$$

\noindent (b) Assume that ${\tt m}=\infty$ (the assumption that the distribution of $\xi$ is nonlattice is not needed). Then
\begin{equation}\label{bla0}
\lim_{t\to\infty} (V(t+h)-V(t))=0.
\end{equation}
\end{lemma}
\begin{proof}
(a) According to Blackwell's theorem,
\begin{equation}\label{blackwell}
\lim_{t\to\infty}(U(t+h)-U(t))={\tt m}^{-1}h.
\end{equation}
In view of \eqref{blackwell}, $\lim_{t\to\infty}(U(t+h-\eta)-U(t-\eta))\1_{\{\eta\leq t-t^{1/2}\}}={\tt m}^{-1}h$ a.s. Recalling \eqref{subad} we infer
$$
\lim_{t\to\infty}\me (U(t+h-\eta)-U(t-\eta))\1_{\{\eta\leq t-t^{1/2}\}}={\tt m}^{-1}h
$$
by Lebesgue's dominated convergence theorem. Another appeal to \eqref{subad} yields $$\me (U(t+h-\eta)-U(t-\eta))\1_{\{t-t^{1/2}<\eta\leq t\}}\leq U(h)\mmp\{t-t^{1/2}<\eta\leq t\},$$ and the right-hand side converges to $0$ as $t\to\infty$. Finally, by monotonicity,
$$\me U(t+h-\eta)\1_{\{t<\eta\leq t+h\}}\leq U(h)\mmp\{t<\eta\leq t+h\},$$ and the right-hand side converges to $0$ as $t\to\infty$. Invoking the first equality in \eqref{eqV} with $x=t$ and $y=h$ completes the proof of part (a).

\noindent (b) If the distribution of $\xi$ is nonlattice, then, by Blackwell's theorem,
\begin{equation}\label{blackwell1}
\lim_{t\to\infty}(U(t+h)-U(t))=0.
\end{equation}
If the distribution of $\xi$ is $d$-lattice, then, by Blackwell's theorem, \eqref{blackwell1} holds for $h=jd$, $j\in\mn$. However, using monotonicity of $U$ we can ensure that \eqref{blackwell1} holds for any fixed $h>0$ in both nonlattice and lattice cases. With this at hand, repeating verbatim the proof of part (a) we arrive at \eqref{bla0}.
\end{proof}

\begin{lemma}\label{key}
Let $f: \mr\to \mr$ be a directly Riemann integrable (dRi) function on $\mr$.

\noindent (a) Assume that ${\tt m}<\infty$ and that the distribution of $\xi$ is nonlattice. Then
$$\lim_{t\to\infty} \int_{[0,\,\infty)} f(t-y){\rm d}V(y)= {\tt m}^{-1} \int_\mr f(y){\rm d}y.$$

\noindent (b) Assume that ${\tt m}=\infty$ (the assumption that the distribution of $\xi$ is nonlattice is not needed). Then
$$\lim_{t\to\infty} \int_{[0,\,\infty)} f(t-y){\rm d}V(y)=0.$$

If $f$ is dRi on $[0,\infty)$ or $(-\infty, 0]$, then the ranges of integration $[0,\,\infty)$ and $\mr$ should be replaced with $[0,\,t]$ and $[0,\infty)$ or $[t,\infty)$ and $(-\infty, 0]$, respectively.
\end{lemma}
\begin{proof}
(a) We only prove the claim under the assumption that $f$ is dRi on $\mr$ which is equivalent to the fact that $f_{+}$ and $f_{-}$ (nonnegative and nonpositive parts of $f$) are dRi on $\mr$. Thus, we can and do assume that $f\geq 0$ on $\mr$. Obviously, it is enough to show that $$\lim_{t\to\infty} \int_{[0,\,t]} f(t-y){\rm d}V(y)= {\tt m}^{-1} \int_0^\infty f(y){\rm d}y$$ and that $$\lim_{t\to\infty} \int_{(t,\,\infty)} f(t-y){\rm d}V(y)= {\tt m}^{-1} \int_{-\infty}^0 f(y){\rm d}y.$$ The proof of the first relation with $U$ replacing $V$ can be found on p.~241--242 in \cite{Resnick:2002}. We only check the second limit relation by following closely aforementioned Resnick's proof.

We proceed via three steps complicating successively the structure of $f$.

\noindent {\sc Step 1}. Suppose first that
$$f(t)=\1_{[(n-1)h,\,nh)}(t),\quad t<0$$ for fixed nonpositive integer $n$ and $h>0$. Then $f(t-y)=1$ if, and only if, $y\in
(t-nh,\,t-(n-1)h]$ which entails $$\int_{(t,\,\infty)}f(t-y){\rm d}V(y)=V(t-(n-1)h)-V(t-nh).$$ By Lemma \ref{bla}(a), the last
difference tends to ${\tt m}^{-1}h$ as $t\to\infty$, thereby proving that $$\lim_{t\to\infty} \int_{(t,\,\infty)}f(t-y){\rm d}V(y)={\tt m}^{-1}h={\tt m}^{-1}\int_{-\infty}^0 f(y){\rm d}y.$$

\noindent {\sc Step 2}. Suppose now that $$f(t)=\sum_{n\leq 0} c_n\1_{[(n-1)h,\,nh)}(t),\quad t<0,$$ where
$\big(c_n\big)_{n\leq 0}$ is a sequence of nonnegative numbers satisfying $\sum_{n\leq 0}c_n<\infty$. An argument similar to that used in the previous step enables us to assert that $$\int_{(t,\,\infty)}f(t-y){\rm d}V(y)=\sum_{n\leq 0}c_n\big(V(t-(n-1)h)-V(t-nh)\big).$$ Using Lemma \ref{bla}(a) in combination with \eqref{subad} we infer with the help of Lebesgue's dominated convergence theorem, $$\lim_{t\to\infty} \int_{(t,\,\infty)} f(t-y){\rm d}V(y)={\tt m}^{-1}h \sum_{n\leq 0} c_n={\tt m}^{-1}\int_{-\infty}^ 0 f(y){\rm d}y.$$

\noindent {\sc Step 3}. Let now $f$ be an arbitrary nonnegative dRi function on $\mr$ (actually, for the present proof it is enough it is dRi on $(-\infty, 0)$). For each $h>0$, put $$\overline{f}_h(t):=\sum_{n\leq 0}\underset{(n-1)h\leq
y<nh}{\sup}\,f(y)\1_{[(n-1)h,\,nh)}(t),\quad t<0$$ and $$\underline{f}_h(t):=\sum_{n\leq 0}\underset{(n-1)h\leq y<nh}{\inf}\,f(y)\1_{[(n-1)h,\,nh)}(t), \ \ t<0.$$ By the definition of direct Riemann integrability, $$\sum_{n\leq 0}\underset{(n-1)h\leq y<nh}{\sup}\,f(y)<\infty \quad \text{and}\quad \sum_{n\leq 0}\underset{(n-1)h\leq y<nh}{\inf}\,f(y)<\infty$$ for each $h>0$. Thus, the functions $\overline{f}_h$ and $\overline{f}_h$ have the same structure as the functions discussed in Step 2. According to the result of Step 2,
$$\lim_{t\to\infty} \int_{(t,\,\infty)}\overline{f}_h(t-y){\rm d}V(y)={\tt m}^{-1}h\sum_{n\leq 0}\underset{(n-1)h\leq y<nh}{\sup}\,f(y)=:{\tt m}^{-1}\overline{\sigma}(h)$$ and $$\lim_{t\to\infty} \int_{(t,\,\infty)}\underline{f}_h(t-y){\rm d}V(y)={\tt m}^{-1}h\sum_{n\leq 0}\underset{(n-1)h\leq y<nh}{\inf}\,f(y)=:\mu^{-1}\underline{\sigma}(h)$$ for all $h>0$.
Since, for each $h>0$, $$\underline{f}_h(t)\leq f(t)\leq \overline{f}_h(t),\quad t<0,$$ it follows that
\begin{eqnarray*}
{\tt m}^{-1}\underline{\sigma}(h)&=&\underset{t\to\infty}{\lim\inf}\,\int_{(t,\,\infty)}\underline{f}_h(t-y){\rm d}V(y)\leq \underset{t\to\infty}{\lim\inf}\,\int_{(t,\,\infty)} f(t-y){\rm
d}V(y)\\&\leq& \underset{t\to\infty}{\lim\sup}\,\int_{(t,\,\infty)}f(t-y){\rm d}V(y)\leq \underset{t\to\infty}{\lim\sup}\,\int_{(t,\,\infty)}\overline{f}_h(t-y){\rm d}V(y)\\&=&{\tt m}^{-1}\overline{\sigma}(h).
\end{eqnarray*}
We have $\lim_{h\to 0+}\,\big(\overline{\sigma}(h)-\underline{\sigma}(h)\big)=0$ by the definition of direct Riemann integrability. Also, it is known that
$\lim_{h\to 0+}\, \overline{\sigma}(h)=\int_{-\infty}^0 f(y){\rm d}y$. Letting $h\to 0+$ in the last chain of inequalities completes the proof of part (a).

\noindent (b) Use part (b) of Lemma \ref{bla} in place of part (a) and proceed as above.
\end{proof}

Sometimes it is the case that the precision of Lemma \ref{key} is not needed. In this situation the following `light' version, borrowed from Lemma 9.1 in \cite{Iksanov+Marynych+Samoilenko:2020}, may suffice.
\begin{lemma}\label{keylight}
Let $f: \mr\to [0,\infty)$ be a dRi function on $\mr$. Then for some $r>0$ and all $x\in\mr$
\begin{equation}\label{bound}
\int_{[0,\,\infty)} f(x-y){\rm d}V(y)\leq r.
\end{equation}

If $f$ is dRi on $[0,\infty)$ or $(-\infty, 0]$, then the range of integration $[0,\,\infty)$ should be replaced with $[0,\,x]$ or $[x,\infty)$ and then \eqref{bound} holds for all $x\geq 0$ or all $x\leq 0$, respectively.
\end{lemma}

\vspace{5mm}

\noindent {\bf Acknowledgement}. The present work was supported by National Research Foundation of Ukraine (project 2020.02/0014 ``Asymptotic regimes of perturbed random walks: on the edge of modern and classical probability'').


\begin{thebibliography}{99}

\footnotesize

\bibitem{Alsmeyer+Iksanov+Marynych:2017} G. Alsmeyer, A. Iksanov and A. Marynych, \textit{Functional limit theorems for the number of occupied boxes in the Bernoulli sieve}. Stoch. Proc. Appl. \textbf{127} (2017), 995--1017.

\bibitem{Asmussen:2003} S. Asmussen, \textit{Applied probability and queues}. 2nd Edition, Springer-Verlag, 2003.

\bibitem{Biggins:1977} J.~D. Biggins, \textit{Chernoff's theorem in the branching random walk}. J. Appl. Probab. \textbf{14} (1977), 630--636.

\bibitem{Biggins:1979} J.~D. Biggins, \textit{Growth rates in the branching random walk}. Z. Wahrscheinlichkeitstheorie Verw. Geb. \textbf{48}
(1979), 17--34.

\bibitem{Biggins:1992} J.~D. Biggins, \textit{Uniform convergence of martingales in the branching random walk}. Ann. Probab. \textbf{20} (1992), 137--151.

\bibitem{Buraczewski+Dovgay+Iksanov:2020} D. Buraczewski, B. Dovgay and A. Iksanov, \textit{On intermediate levels of nested occupancy scheme in random environment generated by stick-breaking I}. Electron. J. Probab. \textbf{25}, paper no. 123, 24 pp.

\bibitem{Carlsson+Nerman:1986} H. Carlsson and O. Nerman, \textit{An alternative proof of Lorden's renewal inequality}. Adv. Appl.
Probab. \textbf{18} (1986), 1015--1016.

\bibitem{Dong+Iksanov:2020} C. Dong and A. Iksanov, \textit{Weak convergence of random processes with immigration at random times}. J. Appl. Probab. \textbf{57} (2020), 250--265.

\bibitem{Duchamps+Pitman+Tang:2019} J.-J. Duchamps, J. Pitman and W. Tang, \textit{Renewal sequences and record chains related to multiple zeta sums}. Trans. Amer. Math. Soc. \textbf{371} (2019), 5731--5755.

\bibitem{Frenk:1982} J.~B.~G. Frenk, \textit{On Banach algebras, renewal measures and regenerative processes}. Stichting Mathematisch Centrum, 1987.

\bibitem{Gnedin+Hansen+Pitman:2007} A. Gnedin, A. Hansen and J. Pitman, \textit{Notes on the occupancy problem with infinitely many boxes: general asymptotics and power laws}. Probab. Surv. \textbf{4} (2007), 146--171.

\bibitem{Holmgren+Janson:2017} C. Holmgren and S. Janson, \textit{Fringe trees, Crump--Mode--Jagers branching processes and $m$-ary search trees}. Probab. Surv. \textbf{14} (2017), 53--154.

\bibitem{Iksanov:2016} A. Iksanov, \textit{Renewal theory for perturbed random walks and similar processes}. Birkh\"{a}user, 2016.

\bibitem{Iksanov etal:2021} A. Iksanov, B. Rashytov and I. Samoilenko, \textit{Renewal theory for iterated perturbed random walks on a general branching process tree: early levels}. In preparation.

\bibitem{Iksanov+Marynych+Samoilenko:2020} A. Iksanov, A. Marynych and I. Samoilenko, \textit{On intermediate levels of nested occupancy scheme in random environment generated by stick-breaking II}. Preprint (2020) available at {\tt https://arxiv.org/abs/2011.12231}.

\bibitem{Iksanov+Pilipenko+Samoilenko:2017} A. Iksanov, A. Pilipenko and I. Samoilenko, \textit{Functional limit theorems for the maxima of perturbed random walks and divergent perpetuities in the $M_1$-topology}. Extremes. \textbf{20} (2017), 567--583.

\bibitem{Iksanov+Rashytov:2020} A. Iksanov and B. Rashytov, \textit{A functional limit theorem for general shot noise processes}. J. Appl. Probab. \textbf{57} (2020), 280--294.

\bibitem{Karlin:1967} S. Karlin, \textit{Central limit theorems for certain infinite urn schemes}. J. Math. Mech. \textbf{17} (1967), 373--401.

\bibitem{Mitov+Omey:2014} K.~V. Mitov and E. Omey, \textit{Renewal processes}. Springer, 2014.

\bibitem{Mohan:1976} N.~R.~Mohan, \textit{Teugels' renewal theorem and stable laws}. Ann. Probab., {\bf 4(5)} (1976),  863--868.

\bibitem{Pitman+Tang:2019} J. Pitman and W. Tang, \textit{Regenerative random permutations of integers}. Ann. Probab.
\textbf{47} (2019), 1378--1416.

\bibitem{Pitman+Yakubovich:2019} J. Pitman and Yu. Yakubovich, \textit{Gaps and interleaving of point processes in sampling from a residual allocation model}. Bernoulli. \textbf{25} (2019), 3623--3651.

\bibitem{Rashytov:2018} B. Rashytov, \textit{Power moments of first passage times for some oscillating perturbed random walks}. Theory Stoch. Proc. \textbf{23(39)} (2018), 93--97.

\bibitem{Resnick:2002} S.~I. Resnick, \textit{Adventures in stochastic processes}. 3rd printing, Birkh\"{a}user, 2002.

\bibitem{Rudin:1962} W. Rudin, \textit{Fourier analysis on groups}. John Wiley \& Sons, 1962.

\bibitem{Sen:1981} P.~K. Sen, \textit{Weak convergence of an iterated renewal process}. J. Appl. Probab. \textbf{18} (1981), 291--296.

\bibitem{Sgibnev:1981} M.~ S. Sgibnev, \textit{Renewal theorem in the case of an infinite variance}. Sib. Math. J. \textbf{22} (1982), 787--796.
\end{thebibliography}
\end{document}